\documentclass[leqno,11pt]{amsart}
\usepackage{amsmath,amstext,amssymb,amsopn,amsthm,mathrsfs,wasysym}
\usepackage{color}

\theoremstyle{plain}

\allowdisplaybreaks

\newtheorem{thm}{Theorem}[section]

\newtheorem{pro}[thm]{Proposition}

\newtheorem{lem}[thm]{Lemma}
\newtheorem{rem}[thm]{Remark}
\newtheorem{defin}[thm]{Definition}
\newtheorem*{ex*}{Example}

\newcommand{\N}{\mathbb{N}}

\def\R{\mathbb{R}^d}
\def\RR{\mathbb{R}^{d+1}}
\def\gr{\nabla}
\def\sgr{\nabla_{\!0}}
\def\sgrt{\widetilde{\nabla}_{\!0}}

\def\z{\zeta}
\def\a{\alpha}
\def\b{\beta}

\def\eps{\varepsilon}
\def\vp{\varphi}
\def\Lu{\mathcal{L}_{\mu}}
\def\eL{L_{\mu}}
\def\Q{Q_{n,j,\kappa}^{\mu}}
\def\L2{L^2(B,dW_\mu)}
\def\W{dW_\mu}
\def\Wt{\widetilde{W}_\mu}
\def\dWt{d\widetilde{W}_\mu}

\def\dist{\rho}
\def\distB{\rho_B}
\def\distI{\rho_I}
\def\lambdamu{\lambda_{\mu}}

\DeclareMathOperator{\supp}{supp}
\DeclareMathOperator{\interior}{int}
\DeclareMathOperator{\loc}{loc}

\DeclareMathOperator{\esssup}{esssup}

\title[Analysis in the multi-dimensional ball]{Analysis in the multi-dimensional ball}

\author[P. Sj\"ogren]{Peter Sj\"ogren}
\address{Peter Sj\"ogren, \newline
			Department of Mathematical Sciences, 
			Chalmers University of Technology and University of Gothenburg, \newline
SE-412 96 G\"oteborg, Sweden 
      }
\email{peters@chalmers.se}

\author[T.Z. Szarek]{Tomasz Z. Szarek}
\address{Tomasz Z. Szarek,     \newline
			Institute of Mathematics,
		Polish Academy of Sciences, \newline
      \'Sniadeckich 8,
      00--656 Warszawa, Poland \newline
\indent and \newline
			Department of Mathematics, Informatics and Mechanics,
		University of Warsaw, \newline
		Banacha 2, 
		02--097 Warszawa, Poland
      }
\email{szarektomaszz@gmail.com}

\begin{document}

\begin{abstract}
We study the heat semigroup maximal operator associated with a well-known orthonormal system in the $d$-dimensional ball. The corresponding heat kernel is shown to satisfy Gaussian bounds. As a consequence, we can prove weighted $L^p$ estimates, as well as some weighted inequalities in mixed norm spaces, for this maximal operator. 
\end{abstract}

\maketitle

\footnotetext{
\indent	
The second author was partially supported by the National Science Centre of Poland, project no.\ 2015/19/D/ST1/01178.
}

\section{Introduction and Preliminaries} \label{sec:intro}

Let $|\cdot|$ denote the classical distance on the Euclidean space $\R$, $d \ge 2$, and let $B=\{x \in \R : |x| \le 1\}$ be the closed unit ball. For a fixed $\mu>-1/2$, we consider the second order differential operator
\[
\eL f = - \Delta f 
+ 
\sum_{i,j=1}^d x_i x_j \frac{\partial^2 f}{ \partial x_i \, \partial x_j }
+
(2\mu + d) \sum_{i=1}^d x_i \frac{\partial f}{ \partial x_i },
\]
defined on $C^2(B)$. As we shall see, $\eL$ is symmetric and non-negative in $\L2$, where $\W$ is the measure on $B$ defined by the density
\[
\W (x) = W_\mu(x) \, dx, \qquad
W_\mu(x) = (1-|x|^2)^{\mu - 1/2}, \qquad x \in B.
\]
There are several orthonormal bases in $\L2$ consisting of polynomials which are eigenfunctions of $\eL$, see \cite[pp. 38-40]{DX}. One of them is
\[
\Q(x) = \big( C_{n,j}^\mu \big)^{-1} P_j^{\mu-1/2, n-2j + d/2 -1}(2|x|^2 - 1)
S_{n-2j, \kappa}(x), 
\]
where $n \in \N$, \, $0 \le 2j \le n$ and $1 \le \kappa \le h(n-2j)$.
Here $P_j^{\a,\b}$ is the classical Jacobi polynomial of degree $j$ and type $(\a,\b)$ defined on the interval $[-1,1]$ by the Rodrigues formula, cf.\ \cite[(4.3.1)]{Sz},
\[
P_j^{\a,\b}(x) = \frac{(-1)^j}{2^j  j!} (1-x)^{-\a} (1+x)^{-\b} 
\frac{d^j}{dx^j} \Big[ (1-x)^{j + \a} (1+x)^{j + \b}  \Big],
\qquad j \in \N.
\]
For each $k \in \N=\{0, 1, \ldots \}$, $\{ S_{k,\kappa} : \kappa =1, \ldots, h(k) \}$ is a basis of the solid spherical harmonics of degree $k$ in $\R$, which is normalized in $L^2$ with respect to the area measure on the unit sphere and 
$h(k) = \binom{d+k-1}{k} - \chi_{k \ge 2} \binom{d+k-3}{k-2}$ stands for the dimension of this space.
Further, the normalizing constants are given by
\[
C_{n,j}^\mu
=
\bigg( \frac{\Gamma(j + \mu + 1/2) \, \Gamma(n - j + d/2) }
{2\big(n + \mu + (d-1)/2 \big) j! \, \Gamma\big(n - j +\mu + (d-1)/2\big)} \bigg)^{1/2}.
\]
Note that our definition of $\Q$ coincides with those given in \cite{C,DX}. Moreover, we use the same normalization as in \cite{C}, which is slightly different from that in \cite{DX}. Since $\eL + (d + 2\mu -1)^2/4$   
coincides with the operator denoted $\mathcal{L}^{d,\mu}$ in \cite[p.\,1021]{C}, we have 
\[
\eL \Q = \lambda_n^\mu  \Q, \qquad \lambda_n^\mu = n(n + d + 2\mu -1).
\]
As verified in Proposition~\ref{prop:Lsym} below, 
$\eL$ is symmetric on $C^2(B)$, and it is not hard to show that the closure of $\eL$, from now on denoted by $\Lu$, is given by
\begin{align} \label{def:Lser}
\Lu f = \sum_{n=0}^\infty \sum_{j=0}^{\lfloor n/2 \rfloor} \sum_{\kappa=1}^{h(n-2j)} 
\lambda_n^\mu \big\langle f, \Q \big\rangle_{W_\mu} \Q, \qquad 
\big\langle f, \Q \big\rangle_{W_\mu} = \int_B f \, \Q \, \W
\end{align} 
on the natural domain consisting of those $f \in \L2$ for which the above series converges in $\L2$. Here and later on $\lfloor \cdot \rfloor$ denotes the floor function. It is straightforward to check that $\Lu$ is the unique self-adjoint and non-negative extension of $\eL$. The associated heat semigroup is given via the spectral series
\[
H_t^\mu f 
=
e^{-t \Lu} f
= 
\sum_{n=0}^\infty \sum_{j=0}^{\lfloor n/2 \rfloor} \sum_{\kappa=1}^{h(n-2j)} 
e^{-t \lambda_n^\mu} \big\langle f, \Q \big\rangle_{W_\mu} \Q,
\qquad t>0, \quad f \in \L2.
\]
Further, $H_t^\mu$ possesses an integral representation valid, as we will see below, for general functions $f$. In particular, 
\begin{equation} \label{Hint}
H_t^\mu f (x)
= \int_B h_t^\mu (x,y) f(y) \, \W(y), 
\qquad x \in B, \quad t > 0, \quad f \in \L2,  
\end{equation}
where the heat kernel can be expressed as a highly oscillating series
\[
h_t^\mu (x,y)
=
\sum_{n=0}^\infty \sum_{j=0}^{\lfloor n/2 \rfloor} \sum_{\kappa=1}^{h(n-2j)}
e^{-t \lambda_n^\mu} \Q(x) \Q(y), 
\qquad x,y \in B, \quad t >0,
\]
understood in the $L^2$ sense.
Since there exists a constant $C=C(d,\mu) >0$ such that
\begin{equation}\label{Qest}
\big| \Q(x) \big| \le C e^{n}, \qquad x \in B, \quad n \in \N, \quad 0 \le 2j \le n, \quad 1 \le \kappa \le h(n-2j), 
\end{equation}
see Proposition~\ref{prop:Qest} below,
one can easily show that the series defining $H_t^\mu f (x)$ and $h_t^\mu (x,y)$ converge pointwise and produce continuous functions of 
$(t,x) \in (0,\infty) \times B$ and $(t,x,y) \in (0,\infty) \times B^2$, respectively.

The heat semigroup maximal operator is
\[
H_*^\mu f (x) = \sup_{t > 0} |H_t^\mu f (x)|, \qquad x \in B.
\]

Before we state the main results of our paper, we need to 
introduce some more notation.
For $x,y \in B$, we let $ \widetilde x =\big( x, \sqrt{1-|x|^2} \big)$ and
 $ \widetilde y =\big( y, \sqrt{1-|y|^2} \big)$ be the corresponding points in the hemisphere 
$S_+^d = \{ x \in \mathbb{R}^{d+1} : |x| = 1, x_{d+1} \ge 0 \}$. The relevant distance $\distB$ on the ball $B$,  see for instance \cite[p.\,403]{DaiXu}, is defined as the geodesic distance on  $S^d$ between  $ \widetilde x$
and  $ \widetilde y$, denoted  $ \dist(  \widetilde x, \widetilde y)$.
This means that 
\[
\distB(x,y) =  \dist(  \widetilde x, \widetilde y) = \arccos 
\Big(\langle x , y \rangle + \sqrt{1 - |x|^2} \sqrt{1 - |y|^2}\Big).  
\]
The triple $(B,\W,\distB)$ forms a space of homogeneous type in the sense of Coifman and Weiss; see Lemma~\ref{lem:doub} below. For $1 \le p < \infty$ we denote by $A_p^\mu$ the set of $A_p$ weights for this space.

Using \eqref{Qest} it is not hard to show that for any fixed $1 \le p < \infty$ and $w \in A_p^\mu$ there exists a constant $C=C(d,\mu) >0$ such that for $f \in L^p(B, w \, \W)$
\begin{equation*}
\big| \big\langle f, \big| \Q \big| \big\rangle_{W_\mu} \big| \le  C e^{n} \|f\|_{L^p(B, w \, \W)}, 
\qquad n \in \N, \quad 0 \le 2j \le n, \quad 1 \le \kappa \le h(n-2j). 
\end{equation*}
This together with \eqref{Qest} allows one to prove that the series defining $h_t^{\mu} (x,y)$ and 
$H_t^\mu f (x)$ converge pointwise if $f \in L^p(B, w \, \W)$, $w \in A_p^\mu$, $1 \le p < \infty$, and that \eqref{Hint} is satisfied.

The first main result of our paper reads as follows.

\begin{thm}\label{thm:main}
The heat semigroup maximal operator $H_*^\mu$ is bounded on 
$L^p(B, w \, \W)$, $w \in A_p^\mu$, $1<p<\infty$, and bounded from 
$L^1(B, w \, \W)$ to weak $L^1(B, w \, \W)$, $w \in A_1^\mu$.
\end{thm}

The next result deals with mixed norm estimates. Such an estimate was given in this setting by Ciaurri \cite[Theorem 3]{C} for the Poisson integral. Here we use the Muckenhoupt class $A_q(S^{d-1})$, $q>1$, of the $(d-1)$-dimensional unit sphere and also $A_p((0,1), d\lambdamu)$, $p>1$, which is defined in the interval $(0,1)$ with the Euclidean distance and the measure $d\lambdamu$ defined by the density
\[
d\lambdamu (r) = r^{d-1} (1-r^2)^{\mu-1/2} \, dr, \qquad r \in (0,1).
\]
 
\begin{thm}\label{thm:mixed}
Let $1<p, \, q<\infty$ and assume that $v \in A_q(S^{d-1})$ and $u \in A_p((0,1), d\lambdamu)$. Then we have for any measurable function $f$
\begin{align} \nonumber
& \int_0^1 \Big( \int_{S^{d-1}}
| H_*^{\mu} f (rx') |^q v(x') \, d\sigma(x') \Big)^{p/q} 
u(r) \, d\lambdamu (r) \\ \label{mix}
& \qquad
\le C
\int_0^1 \Big( \int_{S^{d-1}}
|f (rx') |^q v(x') \, d\sigma(x') \Big)^{p/q} 
u(r) \, d\lambdamu (r),
\end{align}
where the constant $C>0$ depends only on $d,\mu, p, q, v$ and $u$.
\end{thm}

Notice that because of the well-known subordination principle, Theorems~\ref{thm:main} and \ref{thm:mixed} imply similar results for the maximal operator $P_*^{\mu} f = \sup_{t > 0} |P_t^\mu f |$ based on the Poisson semigroup
\[
P_t^\mu f (x) = \int_0^\infty H_{t^2/(4u)}^\mu f (x) \frac{e^{-u} \, du}{\sqrt{\pi u}},
\qquad x \in B, \quad t>0.
\]

\begin{rem}\label{rem:C}
\textnormal{As verified in Section~\ref{sec:mixed},
Theorem~\ref{thm:mixed} implies \cite[Theorem 3]{C}, and improves it in several ways. Our operator $H_*^{\mu}$ is defined in a classical way whereas Ciaurri considers the supremum of spherical $L^2$ means of a slightly different Poisson semigroup (our maximal operator $H_*^{\mu}$ dominates the one from \cite{C}). This essentially reduces the problem in \cite{C} to the context of one-dimensional Jacobi expansions and some vector-valued inequalities obtained there. Furthermore, we eliminate the restriction on the type parameter $\mu > 0$ appearing in \cite{C} and consider all admissible $\mu > -1/2$ in this setting.}
\end{rem}

To prove Theorems~\ref{thm:main} and \ref{thm:mixed}, we use the following result, which says that the heat kernel $h_t^\mu (x,y)$ has Gaussian bounds.

\begin{thm}\label{thm:heatest} 
There exist constants $C_1,C_2,c_1,c_2 > 0$ depending only on $d$ and $\mu$ such that
\begin{align*}
\frac{C_1}{ W_{\mu} \big( B(x,\sqrt{t}) \big) }
\exp\bigg( -\frac{c_1 \distB(x,y)^2}{t} \bigg)
\le
h_t^\mu (x,y)
\le
\frac{C_2}{ W_{\mu} \big( B(x,\sqrt{t}) \big) }
\exp\bigg( -\frac{c_2 \distB(x,y)^2}{t} \bigg)
\end{align*}
for all $t>0$ and $x,y \in B$.
\end{thm}
 
Observe that since $(B,\W,\distB)$ is a space of homogeneous type, Theorem~\ref{thm:heatest} combined with standard arguments implies that for non-negative functions $f$ the heat semigroup maximal function $H_*^\mu f$ is comparable with the Hardy-Littlewood maximal function. Thus Theorem~\ref{thm:main} follows from Theorem~\ref{thm:heatest}, as a consequence of the general theory of spaces of homogeneous type.

Theorem~\ref{thm:heatest} was proved by Kerkyacharian, Petrushev and Xu very recently in \cite{KPX1} and \cite{KPX2} (in \cite{KPX2} under the restriction $\mu \ge 0$). Nevertheless, we give another proof of this result, which is of independent interest, and which was established independently of \cite{KPX1,KPX2}. Our proof is different from the one in \cite{KPX1} despite the fact that both proofs use Dirichlet form methods. 
Indeed, our approach is more explicit and uses more concrete analysis and geometry. 
We show that the general theory of Dirichlet forms from \cite{CKP} and \cite{HS} applies.
In particular, we determine explicitly the energy measure $\Gamma$ (see \eqref{iden2}) and the intrinsic distance $\distI$ (see Lemma \ref{lem:dist}), and we do not use the result of Gyrya and Saloff-Coste in \cite[Theorem 3.34]{GS}, which is one of the crucial tools in \cite{KPX1}. Finally, we also have explicit formulas for the orthogonal polynomials and the so-called Friedrichs extension of $\eL$.

The paper is organized as follows. In Section~\ref{sec:Dir} we introduce the framework of a Dirichlet form based on $\Lu$, which allows us to reduce the proof of Theorem~\ref{thm:heatest} to computing the corresponding intrinsic distance (Lemma~\ref{lem:dist}) and showing the related Poincar\'e inequality (Theorem~\ref{thm:Poincare}). 
Section~\ref{sec:distance} contains preparatory results which lead to the proof of Lemma~\ref{lem:dist}, and 
Section~\ref{sec:Poincare} is devoted to the proof of Theorem~\ref{thm:Poincare} and some technical lemmas needed there. In Section~\ref{sec:mixed} we prove Theorem~\ref{thm:mixed} and verify that it implies \cite[Theorem 3]{C}.
Finally, in the appendix we prove the estimate \eqref{Qest}.

\textbf{Notation.}
The terminology pertaining to the Dirichlet forms is adopted from \cite{FOT}; the only exception here is a different normalization of the energy measure, taken from \cite[p.\,1441]{HS} (see also the statement of Lemma~\ref{lem:Dirform} below). 
For the sake of clarity, we now explain some
symbols.
In this paper we consider only real-valued functions. In particular, by  
$\langle f,g \rangle_{W_\mu}$ we mean $\int_{B} f(x) g(x) \, \W (x)$ whenever the integral makes sense. 

Furthermore, we denote
\begin{align*}
S^{d-1} & = \{ x \in \R : |x| = 1 \}, 
\qquad \textrm{($(d-1)$-dimensional unit sphere)}\\
S^d & = \{ x \in \mathbb{R}^{d+1} : |x| = 1 \}, 
\qquad \textrm{($d$-dimensional unit sphere)}\\
S^d_+ & = S^d \cap \{ x \in \mathbb{R}^{d+1} : x_{d+1} \ge 0 \}, 
\qquad \textrm{(hemisphere)}\\
\langle \cdot , \cdot \rangle & \equiv \textrm{inner product in $\R$ or $\RR$ (this will be clear from the context)},\\
\dist(x,y) & = \arccos \langle x , y \rangle, \qquad x,y \in S^d, \quad 
\textrm{(geodesic distance on $S^d$)}\\
\distB(x,y) & =  \arccos 
\Big(\langle x , y \rangle + \sqrt{1 - |x|^2} \sqrt{1 - |y|^2}\Big), \qquad x,y \in B, \quad 
\textrm{(geodesic distance on $B$)}\\
c(x,r) & = \{ y \in S^d : \dist(x,y) < r\}, \qquad x \in S^d, \quad r>0, \quad 
\textrm{(spherical cap)}\\
c^+(x,r) & = c(x,r) \cap S^d_+, \qquad x \in S^d, \quad r>0,\\
\interior B & = \{ x \in \R : |x| < 1 \}, 
\qquad \textrm{($d$-dimensional open unit ball)}\\
B(x,r) & = \{ y \in B : \distB(x,y) < r\}, \qquad x \in B, \quad r>0, \quad 
\textrm{(ball with respect to $\distB$)}\\
e_j & \equiv \textrm{$j$th coordinate vector},\\
\partial_j & = \partial \slash \partial x_j, \qquad 
	\textrm{($j$th partial derivative)}\\
\gr & \equiv \textrm{gradient in $\R$ or $\RR$ (this will be clear from the context)},\\
\sgr & \equiv \textrm{spherical gradient on $S^{d-1}$ or in $\R$},\\
\sgrt & \equiv \textrm{spherical gradient on $S^{d}$},\\
d\sigma & \equiv \textrm{area measure on $S^{d-1}$ or $S^{d}$ 
(this will be clear from the context)}.
\end{align*}

When writing estimates, we will use the notation $X \lesssim Y$ to
indicate that $X \le C Y$ with a positive constant $C$ depending only on $d$ and $\mu$. 
We shall write $X \simeq Y$ when simultaneously $X \lesssim Y$ and $Y \lesssim X$.

\textbf{Acknowledgments.} The authors would like to thank Luz Roncal for drawing our attention to the article \cite{Duo1}. This allowed us to prove Theorem~\ref{thm:mixed} via an extrapolation theorem.

\vspace{0,5cm}
\section{Dirichlet forms approach} \label{sec:Dir}
In this section we show that $\Lu$ is the so-called Friedrichs extension of $\eL$. 
We start by defining the radial derivative and the spherical gradient, and determine the corresponding adjoint operators. 
For $f \in C^1(B)$ we define
\begin{align}\label{radder}
\partial_r f(x) = \sum_{j=1}^d \frac{x_j}{|x|} \, \partial_j f(x), 
\qquad x \in B\setminus \{0\}.
\end{align}
Clearly 
\[
\partial_r f(r \xi) = 
\frac \partial   {\partial r}\, f(r\xi),
\qquad 0<r \le 1, \quad \xi \in S^{d-1},
\]
and for any $\vp \in C_c^1(\interior B)$ integration by parts gives
\begin{align*}
&\int_B \partial_r f(x) \vp(x) \, dx \\
&  \quad =
\int_0^1 r^{d-1} \int_{S^{d-1}} \partial_r \big[ f(r\xi) \big] \vp(r\xi) 
\, d\sigma(\xi) \, dr \\
&  \quad =
\int_{S^{d-1}} \bigg[ r^{d-1} f(r\xi) \vp(r\xi) \Big|_{r=0}^{r=1}  
- \int_0^1 \Big( r^{d-1} \partial_r \big[ \vp(r\xi) \big] + (d-1) r^{d-2} \vp(r\xi)  \Big) f(r\xi) \, dr 
\bigg] d\sigma(\xi).
\end{align*}
Since $\vp = 0$ on $S^{d-1}$, the integrated term vanishes and we obtain 
\begin{align} \label{iden20}
\int_B \partial_r f(x) \vp(x) \, dx
=
- \int_B f(x) \Big( \partial_r \vp(x) + \frac{d-1}{|x|} \vp(x) \Big) \, dx,
\qquad f \in C^1(B), \quad \vp \in C_c^1(\interior B).
\end{align}

For $f \in C^1(B)$, the spherical gradient can be defined by
\begin{align}\label{iden3}
\sgr f(x) = |x| \gr f(x) - x \partial_r f(x), \qquad x \in B \setminus \{ 0 \};
\end{align}
cf.\  \cite[(1.4.6)]{DaiXu}. Notice that the values of $\sgr f$ on a sphere
$\{|x| = r \}$ depend only on the restriction of $f$ to this sphere and that
 $\sgr$ can be seen as vector fields on these spheres. Integration by parts
leads to the identity 
\begin{equation}
  \label{intpart}
 \int_B \sgr f(x) \vp(x) \, dx
=
- \int_B f(x) \Big( \sgr \vp(x) - (d-1)\frac{x}{|x|} \vp(x) \Big) \, dx,
\qquad f, \vp \in C^1(B);
\end{equation}
see  \cite[(1.8.13)]{DaiXu}.

The formulas \eqref{iden20} and \eqref{intpart} motivate the following
definitions of the $L^2$ distribution extensions of the radial derivative $\partial_r$ and the spherical gradient $\sgr$.

\begin{defin}[$L^2$ radial derivative $\partial_r$]\label{def:r}
Let $f \in L^2_{\loc}(\interior B \setminus \{0\})$. We say that $\partial_r f$ exists and belongs to $L^2_{\loc}(\interior B \setminus \{0\})$ if there exists an $h \in L^2_{\loc}(\interior B \setminus \{0\})$ such that
\begin{align*}
\int_B h(x) \vp(x) \, dx
=
- \int_B f(x) \Big( \partial_r \vp(x) + \frac{d-1}{|x|} \vp(x) \Big) \, dx,
\qquad \vp \in C_c^1(\interior B \setminus \{ 0 \}).
\end{align*}
Since $h$ is unique if it exists, 
we denote it by $\partial_r f$.
\end{defin}

\begin{defin}[$L^2$ spherical gradient $\sgr$ in $B$] \label{def:sgr}
Let $f \in L^2_{\loc}(\interior B \setminus \{0\})$. We say that $\sgr f$ exists and belongs to $L^2_{\loc}(\interior B \setminus \{0\})$ if there exists a vector $h=(h_1,\ldots , h_d)$ such that $|h| \in L^2_{\loc}(\interior B \setminus \{0\})$ and 
\[
\int_B h(x) \vp(x) \, dx
=
- \int_B f(x) \Big( \sgr \vp(x) - (d-1)\frac{x}{|x|} \vp(x) \Big) \, dx,
\qquad \vp \in C_c^1(\interior B \setminus \{0\}).
\]
We write $\sgr f$ for $h$.
\end{defin}

We prove the following simple fact.
\begin{pro}\label{prop:Lsym}
$\eL$ is symmetric and non-negative on $C^2(B) \subset \L2$, and for $f,g \in C^2(B)$ we have
\begin{align}\label{iden1}
\langle \eL f ,g \rangle_{W_\mu} 
=
\int_B \Big( (1-|x|^2) \partial_r f(x) \partial_r g(x) + 
\frac{1}{|x|^2} \sgr f(x) \cdot \sgr g(x) \Big)
\, \W(x).
\end{align}
\end{pro}

\begin{proof}
Clearly, it suffices to show \eqref{iden1}. From \cite[(1.4.2)]{DaiXu} and \eqref{radder} we have the identities 
\begin{align*}
\Delta h & = 
\partial_r^2 h + \frac{d-1}{r} \partial_r h + \frac{1}{r^2} \Delta_{0} h,
\qquad h \in C^2(B),\\
\sum_{i,j=1}^d x_i x_j \partial_i \partial_j h(x) \Big|_{x=r\xi} & =
r^2  \partial_r^2 h (r\xi), 
\qquad h \in C^2(B),
\qquad 0<r \le 1, \quad \xi \in S^{d-1},
\end{align*}
where $\Delta_{0}$ is the 
spherical part of the Laplacian in $\R$. Note that $\Delta_{0} f (r\xi)$ is the same as what we get if the 
Laplace-Beltrami operator on $S^{d-1}$, also denoted by $\Delta_{0}$, act in the $\xi$ variable on $f (r\xi)$.
Using the above formulas we obtain
\[
\eL f(r\xi) = 
- \partial_r^2 f (r\xi) - \frac{d-1}{r} \partial_r f (r\xi) 
- \frac{1}{r^2} \Delta_{0} f (r\xi) + r^2  \partial_r^2 f (r\xi)
+ (2\mu + d) r  \partial_r f (r\xi).
\]
Consequently, integrating in spherical coordinates we get
\begin{align*}
\langle \eL f ,g \rangle_{W_\mu} &  =
\int_0^1 \int_{S^{d-1}} r^{d-1} \Big( 
-(1-r^2) \partial_r^2 f (r\xi) - \frac{d-1}{r} \partial_r f (r\xi) 
- \frac{1}{r^2} \Delta_{0} f (r\xi) \\
& \qquad \qquad \qquad \qquad \qquad
+ (2\mu + d) r  \partial_r f (r\xi)
\Big) 
g(r\xi) (1-r^2)^{\mu - 1/2} \, d\sigma(\xi) \, dr.
\end{align*}
Integration by parts gives
\begin{align*}
& - \int_0^1 r^{d-1} 
(1-r^2)^{\mu + 1/2} \partial_r^2 f (r\xi) 
g(r\xi)   \, dr \\
& \quad =
\int_0^1 r^{d-1} 
\partial_r f (r\xi) \partial_r g(r\xi)  (1-r^2)^{\mu + 1/2}  \, dr \\
& \qquad + 
\int_0^1 r^{d-1} \Big( \frac{d-1}{r} (1-r^2) - (2\mu + 1)r \Big)
\partial_r f (r\xi) g(r\xi) (1-r^2)^{\mu - 1/2} \, dr.
\end{align*}
This together with the identity, see \cite[(1.8.14)]{DaiXu},
\begin{align*}
\int_{S^{d-1}} \Delta_{0} \vp \, \psi \, d\sigma
=
- \int_{S^{d-1}} \sgr \vp \cdot \sgr \psi \, d\sigma, 
\qquad \vp,\psi \in C^2(S^{d-1}),
\end{align*}
leads to
\begin{align*}
\langle \eL f ,g \rangle_{W_\mu} &  =
\int_0^1 \int_{S^{d-1}} r^{d-1} \partial_r f (r\xi) \partial_r g (r\xi) 
(1-r^2)^{\mu + 1/2} \, d\sigma(\xi) \, dr\\
& \quad + 
\int_0^1 \int_{S^{d-1}} r^{d-3} \sgr f (r\xi) \cdot \sgr g (r\xi) 
(1-r^2)^{\mu - 1/2} \, d\sigma(\xi) \, dr.
\end{align*}
The identity \eqref{iden1} is justified, and the proof of Proposition~\ref{prop:Lsym} is complete.
\end{proof}

We intend to construct a Dirichlet form based on $\eL$ to which the framework described in \cite[Section 1.2]{CKP} and \cite[Section 2]{HS} applies. Let
\[
\mathcal{E} (f,g) := \langle \eL f , g \rangle_{W_{\mu}},
\qquad D( \mathcal{E} ) = C^2(B) \subset \L2.
\] 
Then we know from \cite[Exercise 1.1.2]{FOT} that 
$(\mathcal{E},  D( \mathcal{E} ) )$ is a closable symmetric form on $\L2$.
Moreover, using Proposition~\ref{prop:Lsym}, it can easily be proved that
\[
\overline{ \mathcal{E} } (f,g)
=
\int_B \Big( (1-|x|^2) \partial_r f(x) \partial_r g(x) + 
\frac{1}{|x|^2} \sgr f(x) \cdot \sgr g(x) \Big)
\, \W(x)
\]
defined on
\begin{align*}
D( \overline{ \mathcal{E} } )
=
&\Big\{ f \in \L2 : \sqrt{1 - |x|^2} \, \partial_r f, \frac{1}{|x|} |\sgr f|  \in \L2 \textrm{ and there exist $f_n \in C^2(B)$} \\
& \,\, \textrm{such that } f_n \to f,
\sqrt{1 - |x|^2} \, \partial_r (f_n - f) \to 0 \textrm{ and } \frac{1}{|x|}|\sgr(f_n - f)| \to 0 \textrm{ in $\L2$}
\Big\}
\end{align*}
is the closure of $(\mathcal{E},  D( \mathcal{E} ) )$; here the derivatives are taken in the $L^2$ sense, see Definitions~\ref{def:r} and \ref{def:sgr}.
By virtue of \cite[Theorem 1.3.1]{FOT}, the closed symmetric form $( \overline{ \mathcal{E} },  D( \overline{ \mathcal{E} } ) )$ gives rise to a self-adjoint operator, say $A$. 
Then, by \cite[Lemma 3.3.1]{FOT} we see that $-A$ 
is a self-adjoint and non-negative extension of $\eL$. It is called the Friedrichs extension. 
On the other hand, $\Lu$ as defined in \eqref{def:Lser} is the unique self-adjoint extension of $\eL$, so it must coincide with the Friedrichs extension of $\eL$.

Next, we show that the symmetric form 
$( \overline{ \mathcal{E} },  D( \overline{ \mathcal{E} } ) )$
on $\L2$ fits into the setting described in \cite[Section 2.1]{HS}.
Observe that $(B,\W)$ is a connected compact separable space and $\W$ is a non-negative Radon measure on $B$ with full support.

\begin{lem}\label{lem:Dirform}
The bilinear symmetric form $( \overline{ \mathcal{E} },  D( \overline{ \mathcal{E} } ) )$
is a regular, strongly local Dirichlet form with the associated energy measure  given by
\begin{align*}
d \Gamma (f,g) & = 
\Gamma (f,g) \, \W, \qquad f,g \in D( \overline{ \mathcal{E} } ),
\end{align*}
where the density is
\begin{align} \label{iden2}
\Gamma (f,g) (x) & = 
(1-|x|^2) \partial_r f(x) \partial_r g(x) 
+ \frac{1}{|x|^2} \sgr f(x) \cdot \sgr g(x), 
\qquad x\in B, \quad f,g \in D( \overline{ \mathcal{E} } ).
\end{align}
Here $d \Gamma (f,g)$ is normalized in a such way that 
$\overline{ \mathcal{E} } (f,g) = \int_B \Gamma (f,g) \, \W$ for $f,g \in D( \overline{ \mathcal{E} } )$.
\end{lem}

This means that we use the normalization of $\Gamma$ from \cite{CKP} and \cite{HS}, which differs slightly from that of \cite{FOT}.

\begin{proof}[Proof of Lemma \ref{lem:Dirform}]
We first show that $( \overline{ \mathcal{E} },  D( \overline{ \mathcal{E} } ) )$
is a regular, strongly local Dirichlet form. Because of 
\cite[Theorem 3.1.2 and Exercise 3.1.1]{FOT}, it suffices to verify that 
$( \mathcal{E} ,  D(  \mathcal{E}  ) )$ is strongly local and Markovian; note that the density property and the Urysohn-type result assumed in 
\cite[Theorem 3.1.2 (i) and (ii)]{FOT} are easy to check in our setting.

\textbf{Strong locality of $( \mathcal{E} ,  D(  \mathcal{E}  ) ) $.}
Let $f,g \in C^2(B)$ and let $U$ be an open neighbourhood of $\supp g$ on which $f$ is constant. Since $\eL f \equiv 0$ on $U$ (because $\partial_j f \equiv 0$ on $U$) and $g \equiv 0$ on $U^c$, we get $\mathcal{E} (f,g) = 0$.

\textbf{Markov property of $( \mathcal{E} ,  D(  \mathcal{E}  ) ) $.}
Let $\phi_\eps \colon \mathbb{R} \to [-\eps, 1 + \eps]$, $\eps > 0$, be  
smooth functions verifying $0 \le \phi_\eps'(t) \le 1$ for $t \in \mathbb{R}$ and
$\phi_\eps (t) = t$ for $t \in [0,1]$. It suffices to show that 
\[
\mathcal{E} (\phi_\eps \circ f, \phi_\eps \circ f)
\le \mathcal{E} (f, f), \qquad \eps > 0, \quad f \in C^2(B). 
\]
Clearly, $\phi_\eps \circ f \in C^2(B)$ and 
\begin{align*}
\partial_r \big( \phi_\eps \circ f \big) (x)
=
\phi_\eps' \big( f(x) \big) \, \partial_r f(x), 
\qquad
\sgr\big( \phi_\eps \circ f \big) (x)
=
\phi_\eps' \big( f(x) \big) \, \sgr f(x), \qquad
x \in B \setminus \{ 0 \}.
\end{align*}
From Proposition \ref{prop:Lsym}, the Markov property follows.

Finally, we consider the energy measure.  
Since 
$\Gamma (f,g)$, as defined in \eqref{iden2}, satisfies
\[
\| \Gamma (f,g) \|_{L^1(B,\W)}
\le
\sqrt{ \overline{ \mathcal{E} } (f,f) \, \overline{ \mathcal{E} } (g,g)}, 
\qquad f, g \in D( \overline{ \mathcal{E} } ),
\]
we see that the bilinear symmetric form 
$D( \overline{ \mathcal{E} } ) \times D( \overline{ \mathcal{E} } ) \ni (f,g)
\mapsto \Gamma (f,g) \in L^1(B,\W)$ is continuous. Thus to finish the proof, it is enough to compute $\Gamma (f,f)$ for $f \in C^2(B)$.
This can be done by checking that $\Gamma$ from \eqref{iden2} satisfies 
\[
\mathcal{E} (uf,u) - \mathcal{E} (u^2,f)/2 = \int_B f \, \Gamma (u,u) \, \W , 
\qquad u,f \in C^2(B);
\]
cf.\,\cite[(3.2.14) and Lemma 3.2.3]{FOT}, or by computing $\Gamma (f,f)$, $f \in C^2(B)$, as described in \cite[p.\,1000]{CKP}.
The details are left to the reader. 
\end{proof}

We define an intrinsic distance on $B$ by
\begin{align*}
\distI (x,y) = \sup \big\{ f(x) - f(y) : f \in D( \overline{ \mathcal{E} } ) \cap C(B), \, \Gamma (f,f) (z) \le 1 \textrm{ for a.a. $z \in B$} \big\},
\qquad x,y \in B.
\end{align*}
The following result will be proved in Section~\ref{sec:distance}.
\begin{lem}\label{lem:dist}
For all $x,y \in B$ we have $\distI (x,y) = \distB (x,y)$.
\end{lem}

In the sequel, we will use the following special case of a simple lemma from \cite{DaiXu}.

\begin{lem}[{\cite[Lemma 11.3.6]{DaiXu}}]\label{lem:doub}
For $x \in B$ and $r > 0$,
\[
W_\mu \big( B(x,r)  \big) \simeq 
\begin{cases}
        r^d \big( \sqrt{1-|x|^2} + r \big)^{2\mu}, & 0 < r \le \pi, \\
        1, &  r > \pi.
    \end{cases}
\]
\end{lem}
This result shows that $(B,\W,\distB)$ possesses the doubling property. 
Because of Lemma~\ref{lem:dist}, $(B,\distI)$ is a complete metric space, and the metric induces the Euclidean topology. 
Furthermore, 
taking Lemma~\ref{lem:Dirform} into account and also the fact that the heat kernel 
$h_t^\mu (x,y)$ is a continuous function of the variables 
$(t,x,y) \in (0,\infty) \times B^2$, one sees that the assumptions stated in \cite[Sections 2.1 and 2.2]{HS} are satisfied. Thus by means of 
\cite[Theorem 2.7]{HS} (see also the last paragraph of \cite[p.\,1000]{CKP}) we conclude that Theorem~\ref{thm:heatest} is equivalent to the scale-invariant Poincar\'e inequality contained in the following result.
\begin{thm}\label{thm:Poincare}
For $f \in D( \overline{ \mathcal{E} } )$ we have
\[
\int_{B(z,r)} | f(x) - f_{B(z,r)} |^2 \, \W (x)
\lesssim
r^2 \int_{B(z,r)}
\Gamma (f,f) (x) \, \W (x),
\qquad z \in B, \quad r > 0,
\]
where 
$f_{B(z,r)} = W_\mu ( B(z,r) )^{-1} \int_{B(z,r)} f \, \W$.
\end{thm}
The proof of Theorem~\ref{thm:Poincare} is rather technical and so postponed until Section~\ref{sec:Poincare}.

\vspace{0,5cm}
\section{Intrinsic distance - proof of Lemma~\ref{lem:dist} }
\label{sec:distance}

We denote by $\pi\colon S^d_+ \to B $ the projection 
$ (x_1, \dots, x_d, x_{d+1}) \mapsto (x_1, \dots, x_d)$. For any function
$f$ defined
in a subset $\Omega$ of  $B$, we set  $\widetilde{f} = f \circ \pi$, defined
in $\pi^{-1}(\Omega) \subset S^d_+$.

\begin{lem}\label{lem:transfer}
\begin{itemize}
\item[(a)]
Let $\Omega$ be a relatively open subset of $B$. If $f \in C^1(\Omega)$,
then  $\widetilde{f} \in C^1(\pi^{-1}(\Omega))$, and
\begin{align}\label{iden4}
\sgrt \widetilde{f} \left(x, \sqrt{1-|x|^2}\right)
= \begin{cases}
        \left( \gr f(x) - x |x| \partial_r f(x) , -   \sqrt{1-|x|^2} \: 
|x|\, \partial_r f(x)  \right), & x \in \Omega \setminus \{0\}, \\
        \left( \gr f(0), 0 \right) , &  x=0  \in \Omega,
    \end{cases}
\end{align}
and 
\begin{align}\label{transfer}
\left| \sgrt \widetilde{f} \left(x, \sqrt{1-|x|^2} \right) \right|^2
= (1-|x|^2) |\partial_r f(x)|^2 + \frac{1}{|x|^2} |\sgr f(x)|^2,  
\qquad x \in \Omega \setminus \{0\}.
\end{align}
\item[(b)]
Let $f \in L^2_{\loc}(\interior B \setminus \{0\})$ and assume that  
$\partial_r f$  and 
 $\sgr f$ exist in the $L^2$ distribution sense in $\interior B \setminus \{0\}$, 
see Definitions~\ref{def:r} and \ref{def:sgr}. 
Further, assume that 
 \begin{equation}
   \label{Y}
  (1-|x|^2)\, |\partial_r f(x)|^2 + \frac{1}{|x|^2}\, |\sgr f(x)|^2 \le 1, \qquad \textrm{a.a.}\; x \in B.
 \end{equation}
Then  $\widetilde{f}$ can, after modification on a null set, be extended to
a continuous function on   $S^d_+$
satisfying a Lipschitz condition 
\begin{equation*} 
  |\widetilde{f}(\xi) - \widetilde{f}(\z)| 
\le \dist(\xi,\z), 
\qquad \xi, \z \in S^d_+.
\end{equation*}
 Further, $f$ then 
 coincides
a.e. with a continuous function defined in $B$ and satisfying
\begin{equation*}
  |f(x) - f(y) | 
\le \distB (x,y), 
\qquad x,  y \in B.
\end{equation*}
\end{itemize}
\end{lem}

\begin{proof}
To  prove  (a), we temporarily let $F(\widetilde x) = f(x)$, where
$\widetilde x = (x, x_{d+1}) \in \Omega \times \mathbb{R}$
and $x = (x_1, \dots, x_d)$. By
  $\widetilde \nabla$ and $\widetilde \partial_r$
we denote the gradient and the radial derivative in $\RR$, respectively. 
Then we apply the 
$(d+1)$-dimensional version of \eqref{iden3} to $F$ at a point 
$\widetilde x \in \pi^{-1}(\Omega)$, getting
\[
\sgrt \widetilde{f} \big(\widetilde x\big) 
=\widetilde \nabla F\big(\widetilde x\big)- 
\widetilde x \,\widetilde \partial_r F\big(\widetilde x\big)
= 
\begin{cases}
        \big(\gr f(x) - x |x| \partial_r f(x), - x_{d+1} |x| \partial_r f(x)\big), & x \in \Omega \setminus \{0\}, \\
        \left( \gr f(0), 0 \right) , &  x=0  \in \Omega.
    \end{cases}
\]
Since $x_{d+1} = \sqrt{{1-|x|^2}}$, this is 
\eqref{iden4}. Then \eqref{transfer} follows easily from \eqref{iden3} and the orthogonality of $\sgr f$ and $x$.

Aiming at (b), for a fixed $0 < \delta < 1/4$ we let 
$\Omega_{\delta} = \{x \in B: \delta < |x| < 1-\delta \}$. We take an approximate identity $(\eta_\eps)_{\eps > 0}$ in $\R$
such that $0 \le \eta_\eps \in C^\infty$ and 
$\mathrm{supp}\:\eta_\eps \subset \{x: |x| < \eps \}$ and 
$\int \eta_\eps \,dx = 1 $ for each $\eps > 0$. 
For $\eps < \delta/2$ define 
$f_\eps = f\ast \eta_\eps $ in $\Omega_{\delta}$. 
Then   $f_\eps \to f$ in $L^2(\Omega_{\delta})$ as $\eps \to 0$ and consequently 
$\widetilde{f_\eps} \to \widetilde{f}$ in $L^2(\pi^{-1}(\Omega_{\delta}))$ as $\eps \to 0$. Further, $f_\eps \in C^1(\Omega_{\delta})$ and applying \eqref{transfer} to $f_\eps$, we obtain
\begin{align}\label{iden21}
\left| \sgrt \widetilde{f_\eps} \left(x, \sqrt{1-|x|^2} \right) \right|
=
\left| \left( \sqrt{1-|x|^2} \, \partial_r f_\eps(x), \frac{1}{|x|} \sgr f_\eps(x) 
\right) \right|,
\qquad x \in \Omega_{\delta}, \quad \eps < \delta/2.
\end{align}
Now we show that for every $0 < \delta < 1/4$ there exists a constant $C_\delta > 0$ such that 
\begin{align}\label{est20}
\big| \sgrt \widetilde{f_\eps} ( \xi ) \big|
\le 1 + C_\delta \eps, 
\qquad \xi \in \pi^{-1}(\Omega_{\delta}), \quad \eps < \delta/2.
\end{align}
First note that the assumption that $\partial_r f$ and $\sgr f$ exist in the $L^2$ distribution sense forces that $\gr f$ exists in the $L^2$ distribution sense and the following hold for a.a. $x \in B$
\begin{align}\label{iden22}
|x| \gr f(x) = \sgr f(x) + x \partial_r f(x), 
\qquad 
\partial_r f(x) = \sum_{j=1}^d \frac{x_j}{|x|} \, \partial_j f(x).
\end{align}
This implies
\[
\partial_j f_\eps (x) = \int_B \partial_j f(x-y) \eta_\eps (y) \, dy, 
\qquad x \in \Omega_{\delta}, \quad \eps < \delta/2, \quad j=1, \ldots,d.
\]
Further, using \eqref{iden22} we get 
\begin{align*}
\sqrt{1-|x|^2} \, \partial_r f_\eps (x)
& =
\int_B \sqrt{1-|x-y|^2} \, \partial_r f(x-y) \eta_\eps (y) \, dy + E_1(x),
\qquad x \in \Omega_{\delta}, \quad \eps < \delta/2, \\
\frac{1}{|x|} \sgr f_\eps(x) 
& =
\int_B \frac{1}{|x-y|} \sgr f(x-y)  \eta_\eps (y) \, dy + E_2(x),
\qquad x \in \Omega_{\delta}, \quad \eps < \delta/2,
\end{align*}
where
\begin{align*}
E_1(x) & = \int_B \sum_{j=1}^d 
\bigg( \frac{\sqrt{1-|x|^2}x_j}{|x|} - \frac{\sqrt{1-|x-y|^2}(x_j-y_j)}{|x-y|} \bigg)
\partial_j f(x-y) \eta_\eps (y) \, dy, \\
E_2(x) &= - \int_B \sum_{j=1}^d 
\bigg( \frac{x x_j}{|x|^2} - \frac{(x-y) (x_j-y_j) }{|x-y|^2} \bigg)
\partial_j f(x-y) \eta_\eps (y) \, dy.
\end{align*}
Observe that there exists a constant $C_\delta > 0$ such that
\[
|(E_1(x), E_2(x))| \le C_\delta \eps, 
\qquad x \in \Omega_{\delta}, \quad \eps < \delta/2.
\]
Indeed, by the assumption \eqref{Y} we get 
$\esssup_{x \in \Omega_{\delta/2}} |\gr f(x)| < \infty$. Further, since the functions $H_j (x) = \frac{\sqrt{1-|x|^2}x_j}{|x|}$ and 
$G_{j,k} (x) = \frac{x_k x_j}{|x|^2}$, $j,k = 1, \ldots, d$, have bounded gradients in $\Omega_{\delta/2}$ (observe that for $x \in \Omega_{\delta}$, $|y| < \eps$ and $t \in [0,1]$ we have $x-ty \in \Omega_{\delta/2}$), the mean value theorem produces the asserted estimate.

Now using \eqref{iden21} and the above estimate for the error term, we see that the left-hand side of \eqref{est20} is bounded by 
\begin{align*}
& \left| \left( \int_B \sqrt{1-|x-y|^2} \, \partial_r f(x-y) \eta_\eps (y) \, dy , \int_B \frac{1}{|x-y|} \sgr f(x-y)  \eta_\eps (y) \, dy
\right) \right|
+ |(E_1(x), E_2(x))| \\
& \qquad \le 
\int_B \left| 
\left( \sqrt{1-|x-y|^2} \, \partial_r f(x-y) , \frac{1}{|x-y|} \sgr f(x-y) \right) \right| \eta_\eps (y) \, dy
+ C_\delta \eps.
\end{align*}
By \eqref{Y} and the fact that $\int \eta_\eps \,dx = 1$, the estimate \eqref{est20} follows in a straightforward way.

Next we show that 
\begin{align}\label{est21}
|\widetilde{f_\eps}(\xi) - \widetilde{f_\eps}(\z)| 
\le 
(1 + C_\delta \eps ) ( \dist(\xi,\z) + 2\pi \delta ), 
\qquad \xi, \z \in \pi^{-1}(\Omega_{2\delta}), 
\quad \eps < \delta/2.
\end{align}
We join $\xi$ and $\z$ by a geodesic in $S^d$. If this geodesic is contained in $\pi^{-1}(\Omega_{2\delta})$, then \eqref{est20} implies that 
$|\widetilde{f_\eps}(\xi) - \widetilde{f_\eps}(\z)| 
\le 
(1 + C_\delta \eps ) \dist(\xi,\z)$.
In the contrary case, the geodesic has an arc in the closed cap $\pi^{-1} (\{x \in B : |x| \le 2 \delta\})$. Then we replace this arc by an arc in the boundary of the cap $\pi^{-1} (\{x \in B : |x| = 2 \delta\})$ of length at most $2\pi \delta$. In this case we get \eqref{est21} as desired.

Now observe that since $\widetilde{f_\eps} \to \widetilde{f}$ in $L^2(\pi^{-1}(\Omega_{2\delta}))$ as $\eps \to 0$, there exists a sequence $\eps_i \to 0$ for which one has a.e. convergence in $\pi^{-1}(\Omega_{2\delta})$. 
Then \eqref{est21} implies that
\begin{equation*}
  |\widetilde{f}(\xi) - \widetilde{f}(\z)| 
\le \dist(\xi,\z) + 2\pi \delta,
\qquad \textrm{a.a.}\; \xi, \, \z \in \pi^{-1}(\Omega_{2\delta}).
\end{equation*}  
Letting now $\delta \to 0$ through a sequence, we conclude that 
\begin{equation} \label{est22}
  |\widetilde{f}(\xi) - \widetilde{f}(\z)| 
\le \dist(\xi ,\z),
\qquad \textrm{a.a.}\; \xi, \, \z \in S^d_+.
\end{equation}
This easily implies that $\widetilde{f}$ can be modified on a null set to become a continuous function verifying \eqref{est22} for all points of $S^d_+$. To obtain the last assertion of (b), we modify $f$ so that $\widetilde{f} = f \circ \pi$.
\end{proof}

Now we are ready to prove Lemma~\ref{lem:dist}.

\begin{proof}[Proof of Lemma~\ref{lem:dist}]
Since the inequality $\distI (x,y) \le \distB (x,y)$ is a straightforward consequence of Lemma \ref{lem:transfer} (b), we focus on the opposite inequality. Let $y \in B$ be fixed and let for $0 < \eps \le \delta$ 
\[
f_{\delta,\eps} (x) = \arccos \frac{\langle x , y \rangle + \sqrt{ (1+\eps)^2 - |x|^2} \sqrt{(1+\eps)^2 - |y|^2}}{(1 + \delta)^2},
\qquad x \in \R, \quad |x| < 1 + \eps.
\]
The numerator here is the scalar product of two vectors in $\mathbb{R}^{d+1}$ of length $1 + \eps$, so it does not exceed $(1+\eps)^2$. Thus $f_{\delta,\eps} \in C^\infty (\{x \in \R : |x| < 1 + \eps\})$, provided that $0 < \eps < \delta$. Assume for a moment that 
\begin{align}\label{est1}
\Gamma (f_{\delta,\eps} , f_{\delta,\eps} ) (x) \le 1,
\qquad x \in B \setminus \{ 0 \}, \quad 0 < \eps < \delta.
\end{align}
Then we have for each $0 < \eps < \delta$, by the definition of $\distI (x,y)$
\begin{align*}
\distI (x,y) 
 \ge 
f_{\delta,\eps} (x) - f_{\delta,\eps} (y).
\end{align*}
Now letting $\delta \to \eps^+$ and then $\eps \to 0^+$ we get 
$\distI (x,y) \ge \distB (x,y)$ for $x \in B$.
Thus the proof of Lemma \ref{lem:dist} will be finished once we verify \eqref{est1}.

Because $f_{\delta,\eps} \in C^1 (B)$ for $0 < \eps < \delta$, we see from Lemma \ref{lem:transfer} (a) that $\widetilde{ f_{\delta,\eps} } \in C^1 (S^d_+)$. To prove \eqref{est1} it is enough to show that
\[
\big| \sgrt \widetilde{f_{\delta,\eps}} (\xi) \big| \le 1,
\qquad \xi \in S^d_+, \quad 0 < \eps < \delta.
\]
This is equivalent to 
\[
\big| \widetilde{f_{\delta,\eps}} (\xi) - 
\widetilde{f_{\delta,\eps}} (\z) \big|
\le 
\dist (\xi, \z), 
\qquad \xi, \z \in S^d_+, \quad 0 < \eps < \delta.
\] 
Now using the simple inequality 
\[
|\arccos (\lambda a) - \arccos (\lambda b)|
\le 
|\arccos (a) - \arccos (b)|,
\qquad |\lambda|, |a|,|b| \le 1;
\]
which can be justified by writing the differences as integrals of the derivatives, we deduce that
\[
\big| \widetilde{f_{\delta,\eps}} (\xi) - 
\widetilde{f_{\delta,\eps}} (\z) \big|
\le 
\big| \widetilde{f_{\eps,\eps}} (\xi) - 
\widetilde{f_{\eps,\eps}} (\z) \big|.
\] 
Thus it suffices to show that
\[
\big| f_{\eps,\eps} (x) - f_{\eps,\eps} (z) \big|
\le 
\distB (x,z), \qquad x,z \in B, \quad \eps > 0.
\]
By definition $f_{\eps,\eps} (x) = \distB(x/(1+\eps), y/(1+\eps))$, and using the triangle inequality for $\distB$ we see that it is enough to prove that
\[
\distB \Big(\frac{x}{1+\eps}, \frac{z}{1+\eps} \Big) \le \distB (x,z), 
\qquad x,z \in B, \quad \eps > 0.
\]
This, in turn, is equivalent to showing that
\begin{align*}
\frac{|x-z|^2}{(1+\eps)^2} +
\left|
\sqrt{1- \left( \frac{|x|}{1+\eps} \right)^2 } - 
\sqrt{1- \left( \frac{|z|}{1+\eps} \right)^2 }
\right|^2
\le 
|x-z|^2 + 
\left| \sqrt{1- |x|^2} - \sqrt{1- |z|^2}
\right|^2,
\end{align*}
for $x,z \in B$ and $\eps > 0$. Finally, this inequality is elementary, since the first term in the left-hand side is controlled by the first term in the right-hand side and similarly for the second terms;
to see the latter it is convenient to use the identity $\sqrt{a} - \sqrt{b} = (a-b)/(\sqrt{a} + \sqrt{b})$.

The proof of Lemma~\ref{lem:dist} is finished.
\end{proof}

\vspace{0,5cm}
\section{Poincar\'e inequality - proof of Theorem~\ref{thm:Poincare}} \label{sec:Poincare}

To begin with, we observe that without any loss of generality we may assume that $0 < r \le \pi$.
Further, because of \cite[Remarks 1 and 2 on p.\,1450]{HS} and the doubling property of $\W$ (see Lemma~\ref{lem:doub}) we may assume that $0 < r \le 1/6$ and replace the left-hand side in Theorem~\ref{thm:Poincare} by 
$\int_{B(z,\tau r)} | f(x) - f_{B(z,\tau r)} |^2 \, \W (x)$ for some fixed $\tau \in (0,1)$.
Furthermore, since $C^2(B)$ is dense in 
$D( \overline{ \mathcal{E} } )$ in the norm generated by 
$\overline{ \mathcal{E} }$ (see \cite[(1.1.1)]{FOT} for the definition of this norm), using standard density arguments we reduce our problem to showing Theorem~\ref{thm:Poincare} for $f \in C^1(B)$.
Finally, we reduce Theorem~\ref{thm:Poincare} to a similar estimate in the context of the hemisphere $S_+^d$, in the following way.

We define the measure $d \Wt$ on $S^d_+$ by the density
\[
d\Wt (x) = \Wt (x) \, d\sigma (x), \qquad
\Wt (x) = x_{d+1}^{2\mu}, \qquad x \in S^d_+.
\]
As easily verified (see for instance \cite[A.5.4]{DaiXu}), we have
\begin{align}\label{iden66}
\int_{S^d_+} g(y) \, \dWt(y)
=
\int_B  g(x,\sqrt{1-|x|^2}) \, \W(x), \qquad g \in L^1(\dWt).
\end{align}
This together with Lemma~\ref{lem:doub} produces
\begin{align}\label{cap}
\Wt \big(  c^+(x,r) \big) \simeq r^d (x_{d+1} + r)^{2\mu},
\qquad 0<r\le \pi, \quad x \in S^d_+.
\end{align}
For $f \in C^1 (B)$ let 
$\widetilde{f}$ be defined as in the beginning of Section \ref{sec:distance}.
In view of Lemma~\ref{lem:transfer} (a) we have $\widetilde{f} \in C^1(S^d_+)$ and 
$\Gamma(f,f)(x) = \big| \sgrt \widetilde{f} \big(x, \sqrt{1-|x|^2} \big) \big|^2$ for $x \in B \setminus \{ 0 \}$.
Applying \eqref{iden66} with $g = \chi_{ c^+(\widetilde{z},r) } | \sgrt \widetilde{f} |^2$ for some 
$\widetilde{z} = \big(z , \sqrt{1 - |z|^2} \big) \in S^d_+$, we see that
\begin{align*}
\int_{B(z,r)} \Gamma (f,f) (x) \, \W (x)
=
\int_{c^+(\widetilde{z},r)}
|\sgrt \widetilde{f} (y) |^2  \, \dWt (y).
\end{align*}
Proceeding in a similar way, we get
\[
f_{B(z,r)} = 
\Wt \big(  c^+(\widetilde{z},r) \big)^{-1} 
\int_{c^+(\widetilde{z},r)} \widetilde{f} \, \dWt
=: \widetilde{f}_{c^+(\widetilde{z},r)}, \qquad r>0, \quad z \in B.
\]
Consequently, using once again \eqref{iden66} we obtain
\begin{align*}
\int_{B(z,r)} | f(x) - f_{B(z,r)} |^2 \, \W (x)
& =
\int_{c^+(\widetilde{z},r)} 
\big| \widetilde{f}(y) - \widetilde{f}_{c^+(\widetilde{z},r)} \big|^2  
\, \dWt (y).
\end{align*}
Thus Theorem~\ref{thm:Poincare} is a consequence of the following result.

\begin{lem}\label{lem:Poin}
There exists $\tau \in (0,1)$ such that
\begin{align}\label{estPoin}
\int_{c^+(z,\tau r)} | f(x) - f_{c^+(z,\tau r)} |^2  \, \dWt (x) 
\lesssim
r^2 \int_{c^+(z,r)}
|\sgrt f (x) |^2  \, \dWt (x), 
\end{align}
for  $z \in S_+^d$, $0 < r \le 1/6$ and $f \in C^1(S^d_+)$.
Moreover, one can take $\tau = \frac{1}{3\pi}$.
\end{lem}

Notice that the left-hand side in \eqref{estPoin} is comparable with
\begin{equation}\label{Poinequiv}
\frac{1}{\Wt \big(  c^+(z,\tau r) \big) } 
\int_{c^+(z,\tau r)} \int_{c^+(z,\tau r)}
|f(x) - f(y)|^2 \, \dWt (y) \, \dWt (x).
\end{equation}
Indeed, inserting the expression for the mean value and then using the Cauchy-Schwarz inequality we see that the left-hand side of \eqref{estPoin} is controlled by \eqref{Poinequiv}. The opposite estimate follows if we write $| f(x) - f(y) | \le  | f(x)- f_{c^+(z,\tau r)}|   +  | f(y)  - f_{c^+(z,\tau r)}|$.
We will use this fact in the proof of Lemma~\ref{lem:Poin} without further mentioning.

In the proof of Lemma~\ref{lem:Poin} we will use the standard Poincar\'e inequality in the unit sphere $S^d$, see \cite[Theorem~5.6.5]{Sa},
\begin{equation}\label{Poinstand}
(r\wedge 1)^{-d} \int_{c(z, r)} \int_{c(z, r)}
|f(x) - f(y)|^2 \, d\sigma (y) \, d\sigma (x)
\lesssim
r^2 \int_{c(z,r)}
|\sgrt f (x) |^2  \, d\sigma (x), 
\end{equation}
holding for $z \in S^d$, $r > 0$ and $f \in C^1(S^d)$.

Further, to prove Lemma~\ref{lem:Poin} we shall need several facts and technical lemmas which are gathered below.
We begin with the following simple relations, see \cite[(A.1.1)]{DaiXu},
\begin{align}\label{distrel}
\frac{2}{\pi} \, \dist(x,y) \le |x - y| \le \dist(x,y), \qquad x,y \in S^d,
\end{align}
which will be frequently used in this section.

\begin{lem}[{\cite[Lemma A.5.4]{DaiXu}}]\label{lem:walk}
Assume that $f\ge 0$ or $f \in L^1(S^d)$.
Then
\[
\int_{S^d} f(y) \, d\sigma (y)
=
\int_{-1}^1 (1-y_{d+1}^2)^{d/2 - 1} \int_{S^{d-1}} 
f \left(\sqrt{1-y_{d+1}^2} \, y', y_{d+1} \right) 
\, d\sigma (y') \, dy_{d+1}.
\]
\end{lem}

\begin{lem}\label{lem:geod}
Assume that the curve $\gamma \colon [a,b] \to S^d_+ \subset \RR$ belongs to $C^1([a,b])$, where $[a,b]$ is a bounded interval. Then for any $f \in C^1(S^d_+)$ we have
\[
|f(\gamma (b)) -  f(\gamma (a))| \le 
\int_a^b \big| \sgrt f \big( \gamma (t) \big) \big| |\gamma' (t)| \, dt.
\]
\end{lem}

\begin{proof}
Let $F$ be an extension of $f$ to 
$E = \{ z \in \RR \setminus \{ 0 \} : z_{d+1} \ge 0 \}$ which is homogeneous of order $0$, i.e., $F(z) = f (z/|z|)$ for 
$z \in E$. 
Then
$F \in C^1(E)$ and $\sgrt f (z) = \gr F (z)$, $z \in S^d_+$, because of \eqref{iden3} in $d+1$ dimensions.
Since $\gamma$ has values in $S^d_+$, we see that $f \circ \gamma = F \circ \gamma \in C^1([a,b])$ and we have
\[
|f(\gamma (b)) -  f(\gamma (a))| 
\le 
\int_a^b |(F \circ \gamma)'(t) | \, dt.
\]
The asserted inequality follows.

\end{proof}

Observe that \eqref{Poinstand} and Lemmas \ref{lem:walk} and \ref{lem:geod} hold also if $d$ is replaced by 1. 
Then the surface measure on  $S^0 = \{-1, 1 \}$  is defined as $d\sigma = \delta_{ \{ -1 \} } + \delta_{ \{ 1 \} }$.  This will be used in the proof of Lemma~\ref{lem:Poin} below.

\begin{proof}[Proof of Lemma~\ref{lem:Poin}]
Let $\tau = \frac{1}{3\pi}$, $z \in S_+^d$, $0 < r \le 1/6$ and $f \in C^1(S^d)$. 
To proceed, it is convenient to distinguish two cases.

\noindent \textbf{Case 1:} $z_{d+1} > 2\tau r$.
In this situation, in view of \eqref{cap} and \eqref{distrel}, we have $\Wt \big(  c^+(z,\tau r) \big) \simeq r^d z_{d+1}^{2\mu}$ and 
\begin{equation*}
\Wt (y) \simeq \Wt (z), \qquad y \in c^+(z,\tau r).
\end{equation*}
Since $c^+(z,\tau r) = c(z,\tau r)$, the required estimate is a direct consequence of the standard Poincar\'e inequality \eqref{Poinstand}.

\noindent \textbf{Case 2:} $0 \le z_{d+1} \le 2 \tau r$.
We write $z = \Big(\sqrt{1-z_{d+1}^2} \, z', z_{d+1} \Big)$ with $z_{d+1} \in [0,1]$ and $z' \in S^{d-1}$, and similarly for $x,y \in S_+^{d}$. Using \eqref{distrel} and the identity
\[
|z-y|^2 = |z' - y'|^2 \sqrt{1-z_{d+1}^2} \, \sqrt{1-y_{d+1}^2}
+ \Big| \sqrt{1-z_{d+1}^2} - \sqrt{1-y_{d+1}^2} \Big|^2
+ |z_{d+1} - y_{d+1}|^2,
\]
it is not hard to show that
\begin{align}\label{rectangle}
c^+(z,\tau r) \subset R(z,\tau r) \subset c^+(z,r),
\end{align}
where 
$R(z,\tau r) = \left\{ \left(\sqrt{1 - y_{d+1}^2} \, y',y_{d+1}\right)\in S_+^d : 0 \le y_{d+1} < 3 \tau r,\, \dist(y',z') < \pi \tau r \right\}$.

Applying the triangle inequality to \eqref{Poinequiv} and using \eqref{cap},
we see that it suffices to verify that
\begin{align} \label{estJ}
& r^{-d - 2\mu} 
\int_{c^+(z,\tau r)} \int_{c^+(z,\tau r)}
\left|f(x) - f\left( \sqrt{1 - x_{d+1}^2} \, y',x_{d+1} \right) \right|^2 
\, \dWt (y) \, \dWt (x) \\ \label{estI}
& \quad + r^{-d - 2\mu} 
\int_{c^+(z,\tau r)} \int_{c^+(z,\tau r)}
\left| f\left( \sqrt{1 - x_{d+1}^2} \, y',x_{d+1} \right) - f(y) \right|^2 
\, \dWt (y) \, \dWt (x) \\ \nonumber
& \qquad \lesssim
r^2 \int_{c^+(z,r)}
|\sgrt f (x) |^2  \, \dWt (x).
\end{align}

We first focus on estimating \eqref{estJ}. 
For $0 < a < 1$ let $f_a$ be the function defined on $S^{d-1}$ by $f_a(y') = f\left( \sqrt{1 - a^2} \, y',a \right)$.
Then $f_a \in C^1(S^{d-1})$, and \cite[Corollary 1.4.3]{DaiXu} implies that
\begin{align*} 
|\sgr f_a (y')| 
\le 
\left|\sgrt f \left( \sqrt{1 - a^2} \, y',a \right) \right|,
\qquad y' \in S^{d-1}, \quad 0 < a < 1.
\end{align*}
Therefore, using Lemma~\ref{lem:walk}, \eqref{rectangle} and then applying the standard Poincar\'e inequality \eqref{Poinstand} in $S^{d-1}$ to $f_{x_{d+1}}$, we get the following estimate of \eqref{estJ}
\begin{align*}
& r^{-d - 2\mu}
\int_0^{3 \tau r}
\int_{ \dist(x',z') < \pi \tau r} \int_{ \dist(y',z') < \pi \tau r} 
|f_{x_{d+1}} (x') - f_{x_{d+1}} (y') |^2 \, d\sigma(y') \, d\sigma(x') 
\, x_{d+1}^{2\mu} \, dx_{d+1} \\
& \qquad \qquad \times
\int_0^{3 \tau r} \, y_{d+1}^{2\mu} \, dy_{d+1} \\
& \quad \lesssim
r^{2} \int_0^{3 \tau r} \int_{ \dist(x',z') < \pi \tau r}
\left| \sgrt f \left( \sqrt{1 - x_{d+1}^2} \, x', x_{d+1} \right) \right|^2
\, d\sigma(x') \, x_{d+1}^{2\mu} \, dx_{d+1}.
\end{align*}
This leads directly to the asserted bound for the quantity in \eqref{estJ}.

Next, we deal with the term \eqref{estI}.
We apply Lemma~\ref{lem:geod} with 
$\gamma(t) = \big( \sqrt{1 - t^2} \, y',t \big)$, 
$t \in [x_{d+1} \wedge y_{d+1}, x_{d+1} \vee y_{d+1}] \subset [0,3 \tau r]$; note that for these $t$ and $x,y \in c^+(z,\tau r)$ obviously $\gamma(t) \in R(z,\tau r) \subset c^+(z,r)$.
Denoting $F = |\sgrt f|^2 \chi_{  c^+(z,r) }$, we apply
the Cauchy-Schwarz inequality, Lemma~\ref{lem:walk}, and \eqref{rectangle}. As a result, we see that the expression in \eqref{estI} is controlled by
\begin{align*}
& r^{-d - 2\mu}
\int_{ c^+(z,\tau r) } \int_{ c^+(z,\tau r) }
|y_{d+1} - x_{d+1}| \int_{x_{d+1} \wedge y_{d+1}}^{x_{d+1} \vee y_{d+1}} 
F\left( \sqrt{1 - t^2} \, y',t \right)  
\, dt \, \dWt(y) \, \dWt(x)\\
& \lesssim
r^{1-d - 2\mu} 
\int_0^{3 \tau r} \int_0^{3 \tau r}
\int_{ \dist(y',z') < \pi \tau r} 
\int_{x_{d+1} \wedge y_{d+1}}^{x_{d+1} \vee y_{d+1}}  F\left( \sqrt{1 - t^2} \, y',t \right) \, dt \, d\sigma(y') \, y_{d+1}^{2\mu} \, dy_{d+1} \, x_{d+1}^{2\mu} \, dx_{d+1} \\
& \qquad \qquad \times
\int_{ \dist(x',z') < \pi \tau r} \, d\sigma(x') \\
& \simeq
r^{- 2\mu} 
\int_0^{3 \tau r} \int_0^{3 \tau r}
\int_{ \dist(y',z') < \pi \tau r} 
\int_{x_{d+1} \wedge y_{d+1}}^{x_{d+1} \vee y_{d+1}}  F\left( \sqrt{1 - t^2} \, y',t \right)  
\, dt \, d\sigma(y') 
\, y_{d+1}^{2\mu} \, dy_{d+1} \, x_{d+1}^{2\mu} \, dx_{d+1}.
\end{align*}
Since the roles of $x_{d+1}$ and $y_{d+1}$ in the last expression are symmetric, we may assume that $x_{d+1} \le y_{d+1}$.
Further, changing the order of integration we see that \eqref{estI} is dominated, up to a multiplicative constant, by
\begin{align*}
& r^{- 2\mu} \int_{ \dist(y',z') < \pi \tau r} 
\int_0^{3\tau r}\left(\int_0^{t} x_{d+1}^{2\mu} \, dx_{d+1} \right)
\left(\int_t^{3 \tau r} y_{d+1}^{2\mu} \, dy_{d+1} \right)
F\left( \sqrt{1 - t^2} \, y',t \right) \, dt \, d\sigma(y') \\
& \lesssim
r^2 \int_{ \dist(y',z') < \pi \tau r} 
\int_0^{3\tau r} F\left( \sqrt{1 - t^2} \, y',t \right) t^{2\mu} \, dt \, d\sigma(y').
\end{align*}
Finally, Lemma~\ref{lem:walk} leads to the required estimate for the expression in \eqref{estI}.

This finishes Case 2, and the proof of Lemma~\ref{lem:Poin} is complete.

\end{proof}

\vspace{0,5cm}
\section{Mixed norm estimates - proof of Theorem~\ref{thm:mixed}}
\label{sec:mixed}
The following lemma deals with products of weights. To state it we write $[\cdot]_p$ for the $A_p$ constants as defined e.g. in \cite[(1.1)]{Duo1}. We recall that $A_p^{\mu} = A_p (B, \W, \distB)$.

\begin{lem}\label{lem:weight}
Let $1<p<\infty$ and $v \in A_p(S^{d-1})$, $u \in A_p((0,1), d\lambdamu )$. Then $w(x) = v(x/|x|) u(|x|) \in A_p^{\mu}$ and
\[
[w]_p \le C [v]_p [u]_p,
\]
where the $A_p$ constants are taken in the relevant spaces, and $C>0$ depends only on $d,\mu$ and $p$.
\end{lem}

We first use Lemma~\ref{lem:weight} to prove Theorem~\ref{thm:mixed}. The proof relies on the extrapolation theorem of Rubio de Francia, see for instance \cite[Theorem 3.1]{Duo1}. Its proof is based on the Rubio de Francia algorithm and easily extends to the setting of a space of homogeneous type. In our context the extrapolation theorem implies that if \eqref{mix} holds for $1< p = q < \infty$, then it holds for all $1< p,q < \infty$. The case $p=q$ is a simple consequence of Theorem~\ref{thm:main} and Lemma~\ref{lem:weight}. This concludes the proof of Theorem~\ref{thm:mixed}.

\begin{proof}[Proof of Lemma~\ref{lem:weight}]
We claim that for any $x\in B$ and $s>0$ there exist a subinterval $I$ of $(0,1)$ and a spherical cap $Q$ in $S^{d-1}$ such that
\begin{align} \label{est7}
B(x,s) \subset P_{I,Q}:=\{rz' : r \in I, z' \in Q\}, \qquad 
\lambdamu(I) \sigma(Q) \lesssim W_\mu(B(x,s)).
\end{align}
Given this claim, the desired conclusion follows from a straightforward computation involving the $A_p$ condition, since the measure $dW_\mu$ is $d\lambdamu \times d\sigma$ when expressed in polar coordinates.

To prove the claim, observe that it is trivial if $s \ge 1/6$, since then $W_\mu(B(x,s)) \simeq 1$ and we can take $I=(0,1)$ and $Q= S^{d-1}$. Thus from now on we assume that $s< 1/6$ and consider three cases.

\noindent \textbf{Case 1:} $\sqrt{1-|x|^2} > 2s$, $|x| \le 2s$.
Since we have $|z-x|<s$ for $z \in B(x,s)$, 
we can take $I=(0,3s)$ and $Q=S^{d-1}$. Then using Lemma~\ref{lem:doub} and the fact that $1-|x|^2 \simeq 1$ we obtain $W_\mu ( B(x,s)) \simeq s^d \simeq \lambdamu(I)$, which leads directly to \eqref{est7}.

\noindent \textbf{Case 2:} $\sqrt{1-|x|^2} > 2s$, $|x| > 2s$.
The ball $B(x,s)$ is the image under the projection $\pi$ of the spherical cap in $S_+^d$ of center 
$\pi^{-1} (x) = (x, \sqrt{1-|x|^2})$ and geodesic radius $s$. The projection on the $x_{d+1}$ axis of this cap is contained in the interval $(\sqrt{1-|x|^2} - s, \sqrt{1-|x|^2} + s)$. Since $s < \sqrt{1-|x|^2}/2$, this implies that
\begin{align*}
\sqrt{1-|x|^2}/2 < 
\sqrt{1-|z|^2} < 3 \sqrt{1-|x|^2}/2, 
\qquad z \in B(x,s).
\end{align*}
The largest and smallest values of  $|z|$ as $z \in \overline{B(x,s)}$ are taken at points $z_+$ and $z_-$ in 
$\partial B(x,s)$ which are multiples of $x$. In the plane spanned by $x$ and the $x_{d+1}$ axis, we consider the chord between $\pi^{-1} (x)$ and $\pi^{-1} (z_+)$ and that between $\pi^{-1} (x)$ and $\pi^{-1} (z_-)$.
The angles between these chords and the hyperplane $\R$ are at least $\arccos \sqrt{1-|x|^2}$ and $\arccos \sqrt{1-|z_-|^2}$, respectively. This implies that 
$|z_+| \le |x| + s \sqrt{1-|x|^2}$ and $|z_-| \ge |x| - s \sqrt{1-|z_-|^2} \ge |x| - 3s \sqrt{1-|x|^2}/2$.
As a result, we see that 
\begin{align} \label{iden89}
\big| |z| - |x| \big| < 3s \sqrt{1-|x|^2}/2, \qquad z \in B(x,s).
\end{align}

Next, we estimate $\dist( x/|x|, z/|z| )$ for $z \in B(x,s)$, which is the angle between the $\R$ vectors $x$ and $z$. Now $B(x,s)$ is contained in the closed $d$-dimensional Euclidean ball $B_E$ of center $x$ and radius $s$. The maximal angle $\theta$ between $x$ and points of $B_E$ satisfies $\sin \theta = s/|x|$. Thus $\theta \le \pi s/(2|x|)$, which means that 
\begin{align} \label{iden90}
\dist( x/|x|, z/|z| ) < \pi s/(2|x|), \qquad z \in B(x,s).
\end{align}

In view of \eqref{iden89} and \eqref{iden90}, we can take
\[
I=\big\{r \in (0,1): \big| r-|x| \big| < 3s \sqrt{1-|x|^2}/2  \big\},
\qquad 
Q = \big\{z' \in S^{d-1}: 
\dist\big(  z' , x/|x| \big) < \pi s/(2 |x|) \big\}.
\]
Since
\[
\lambdamu(I) \simeq s (1-|x|^2)^{\mu} |x|^{d-1}, \qquad
\sigma(Q) \simeq (s/|x|)^{d-1},
\]
this implies \eqref{est7}, in view of Lemma~\ref{lem:doub}.

\noindent \textbf{Case 3:} $\sqrt{1-|x|^2} \le 2s$.
Arguing as in Case 2, we see that $B(x,s)\subset P_{I,Q}$, where 
$I=(\sqrt{1-9s^2},1)$ and $Q=\{z' \in S^{d-1}: \dist(z' , x/|x|) < \pi s\}$.
Then $\lambdamu(I) \simeq s^{2\mu + 1}$, $\sigma(Q) \simeq s^{d-1}$, and the estimate in \eqref{est7} follows directly from Lemma~\ref{lem:doub}.

The proof of Lemma~\ref{lem:weight} is finished.

\end{proof}

Finally, we verify that Theorem~\ref{thm:mixed} implies \cite[Theorem 3]{C} as claimed in Remark~\ref{rem:C}.
Using the subordination formula, we see that the Poisson maximal operator considered in \cite{C} is dominated by $H_*^{\mu}$. Further, in Theorem~\ref{thm:mixed} we can put $v\equiv 1$ and replace $u$ by a weight satisfying the conditions imposed by Ciaurri. Indeed, his conditions say, in our notation, that
\[
\sup_{0<x<y<1} \frac{(1-x)^{p/2}}{(y-x)^p} 
\bigg( \int_x^y u(r) \a_p(r) \, dr \bigg)
\bigg( \int_x^y u(r)^{1-p'} \b_p(r) \, dr \bigg)^{p-1} 
< \infty,
\]
where $1/p + 1/p' = 1$ and 
\[
\a_p(r) = r^{(d-1)(1-p/2)} (1-r)^{\mu (1-p/2) -1/2}, \qquad
\b_p(r) = \big( (1-r)^{1/2} \a_p(r) \big)^{1-p'} (1-r)^{-1/2}.
\] 
For $\mu \ge 0$ the above property forces 
\begin{align} \nonumber
& \sup_{0<x<y<1} \lambdamu(x,y)^{-p} 
\bigg( \int_x^y u(r) \a_p(r) r^{(d-1)p/2} (1-r)^{\mu p/2} \, dr \bigg) 
\\ \label{condition}
& \qquad \qquad \qquad \qquad \qquad \qquad \times
\bigg( \int_x^y u(r)^{1-p'} \b_p(r) 
r^{(d-1)p'/2} (1-r)^{\mu p'/2} \, dr \bigg)^{p-1} 
< \infty,
\end{align}
because here $r\le y$, $1-r\le 1-x$, and 
$\lambdamu(x,y) \simeq (y-x)(1-x)^{\mu-1/2} y^{d-1}$.
Since \eqref{condition} is equivalent to $u \in A_p((0,1), d\lambdamu)$, we see that Theorem~\ref{thm:mixed} implies \cite[Theorem 3]{C}.

\vspace{0,5cm}
\section{Appendix}
\label{sec:appendix}
Here we verify the estimate \eqref{Qest}, which is restated in the following proposition.
\begin{pro}\label{prop:Qest}
For $n \in \N$, $0 \le 2j \le n$ and $1 \le \kappa \le h(n-2j)$
\begin{equation*}
\big| \Q(x) \big| \lesssim e^{n}, \qquad x \in B. 
\end{equation*}
\end{pro}

To prove this, we will use the following simple result.
\begin{lem}\label{lem:expest}
\begin{itemize}
\item[(a)] For any $a > 0$ we have
\[
x^x \le (x+a)^x \le e^a x^x, \qquad x > 0.
\]
\item[(b)] The following estimate holds
\[
(x+y)^{x+y} \le 2^{x+y} x^x y^y, \qquad x,y > 0.
\]
\end{itemize}
\end{lem}

\begin{proof}
Since the function $s \mapsto \left(1 + \frac{1}{s} \right)^s$ is increasing for $s>0$ and has limit $e$ as $s \to \infty$, item (a) follows from the identity $(x+a)^x = x^x (1+\frac{a}{x})^{\frac{x}{a} a}$.
We focus on showing (b). Without any loss of generality we may assume that $y \ge x$, i.e. $y=\lambda x$ for some $\lambda \ge 1$. 
Using this we can easily see that our task is equivalent to showing that
\[
(1+\lambda)^{1+\lambda} \le 2^{1+\lambda} \lambda^\lambda, \qquad \lambda \ge 1.
\]
Letting $F(\lambda) = 2^{1+\lambda} \lambda^\lambda (1+\lambda)^{-1-\lambda}$, we must prove that $F(\lambda) \ge 1$ for $\lambda \ge 1$. This follows from the fact that $F(1) = 1$ and $F'(\lambda) = F(\lambda) \log \frac{2 \lambda}{1 + \lambda} \ge 0$ for $\lambda \ge 1$.
\end{proof}

Stirling's formula implies that for a fixed $x_0 > 0$ we have
\begin{equation}\label{Gasy}
\Gamma(x) \simeq x^{x-1/2} e^{-x}, \qquad x\ge x_0,
\end{equation}
where the implicit constants depend on $x_0$.

\begin{proof}[Proof of Proposition~\ref{prop:Qest}]
We first show that $\big( C_{n,j}^\mu \big)^{-2}$ and $\sup_{x \in B} |S_{n-2j, \kappa}(x)|$ have at most polynomial growth in $n$. For the latter quantity, this follows from the formula
\[
\sum_{\kappa = 1}^{h(n-2j)} |S_{n-2j, \kappa}(x)|^2 
= \frac{h(n-2j)}{\sigma(S^{d-1})} |x|^{2n - 4j} 
\lesssim h(n-2j) \le (n+1)^d,
\qquad x \in B,
\]
which can be deduced for instance from \cite[Corollary 1.2.7]{DaiXu} (note that here we use another normalization of spherical harmonics than in \cite{DaiXu}).

Now we focus on $C_{n,j}^\mu$. Using \eqref{Gasy} and then Lemma~\ref{lem:expest} (a), we obtain
\begin{align*}
\big( C_{n,j}^\mu \big)^{-2}
& = 
\frac{2\big(n + \mu + (d-1)/2 \big) \, \Gamma(j + 1) \, \Gamma\big(n - j +\mu + (d-1)/2\big)} 
{\Gamma(j + \mu + 1/2) \, \Gamma(n - j + d/2) } \\
& \simeq
(n+1) \,
\frac{(j+1)^{j+1/2} \big(n - j +\mu + (d-1)/2\big)^{n - j +\mu + d/2 - 1} }
{(j + \mu + 1/2)^{j + \mu} (n - j + d/2)^{n - j + (d-1)/2}} \\
& \lesssim
(n+1)^A \,
\frac{(j+1)^{j} \big(n - j +\mu + (d-1)/2\big)^{n - j} }
{(j + \mu + 1/2)^{j} (n - j + d/2)^{n - j}}
\simeq
(n+1)^A, 
\end{align*}
uniformly in $n \in \N$ and $0 \le 2j \le n$, where $A=A(d,\mu)$ is some positive constant.

Finally, we show that 
\[
\sup_{y \in [-1,1]} \big| P_j^{\mu-1/2, n-2j + d/2 -1}(y) \big|
\lesssim
2^{n}, 
\qquad n \in \N, \quad 0 \le 2j \le n.
\]
From \cite[(7.32.2)]{Sz} we see that the left-hand side above is equal to $\binom{j+q}{j}$, where $q = (\mu-1/2) \vee (n-2j + d/2 -1) \ge 0$.
Combining this with \eqref{Gasy}, Lemma~\ref{lem:expest} (b) 
(specified to $x = j + 1/2$, $y = q + 1/2$) and (a), we get
\begin{align*}
\sup_{y \in [-1,1]} \big| P_j^{\mu-1/2, n-2j + d/2 -1}(y) \big|
& = 
\binom{j+q}{j} 
= \frac{\Gamma(j + q + 1)}{\Gamma(j + 1) \, \Gamma(q + 1)}
\simeq
\frac{(j + q + 1)^{j + q + 1/2}}{(j + 1)^{j + 1/2} (q + 1)^{q + 1/2}} \\
& \le 
2^{j + q + 1}
\frac{(j + 1/2)^{j + 1/2} (q + 1/2)^{q + 1/2}}
{(j + 1)^{j + 1/2} (q + 1)^{q + 1/2}}
\simeq 
2^{j + q}
\lesssim 
2^n.
\end{align*}
This finishes the proof of Proposition~\ref{prop:Qest}.
\end{proof}


\end{document}